\documentclass[12pt,draft,amsmath,amscd,amsbsy,amssymb]{amsart}
 \relax

\chardef\bslash=`\\ 

\makeatletter \def\verbatim{\interlinepenalty\@M \@verbatim
\leftskip\@totalleftmargin\advance\leftskip2pc
\frenchspacing\@vobeyspaces \@xverbatim} \makeatother \hfuzz1pc

\makeatletter \def\dgt@k{\dg@DX=-3 \dg@DY=2 \dg@SIZE=3}
\makeatother

\makeatletter \def\dgt@kk{\dg@DX=3 \dg@DY=-1 \dg@SIZE=3}

\theoremstyle{plain} \newtheorem{thm}{Theorem}[section]

\theoremstyle{definition}

\usepackage[all]{xy}

\begin{document}

\title[Strong topology on the set of persistence diagrams]
{Strong topology on the set of persistence diagrams}
\author{Volodymyr Kiosak}  
\address{Institute of Engineering
Odessa State Academy of Civil Engineering and Architecture
Didrihson  st.,   4,  Odessa,  65029, Ukraine}
\email{kiosakv@ukr.net}

\author{Aleksandr Savchenko}
\address{Kherson State Agrarian University, Stretenska  st.,  23,  Kherson, 73006, Ukraine}
\email{savchenko.o.g@ukr.net}

\author{Mykhailo Zarichnyi}
\address{Faculty of Mathematics and Natural Sciences,
University of Rzesz\'ow, 1 Prof. St. Pigo\'n Street
35-310, Rze\-sz\'ow, Poland}
\email{zarichnyi@yahoo.com}
\thanks{}
\subjclass[2010]{55N35, 54E35} 
\keywords{persistence diagram, bottleneck distance}


\begin{abstract} We endow the set of persistence diagrams with the strong topology (the topology of countable direct limit of increasing sequence of bounded subsets considered in the bottleneck distance). The topology of the obtained space is described.

Also, we prove that the space of persistence diagrams with the bottleneck metric has infinite asymptotic dimension in the sense of Gromov.
\end{abstract}

\maketitle
\section{Introduction}
Topological Data Analysis (TDA) is a field in applied mathematics concentrated around investigation of big data by topological methods. Imposing metric structures in the data set allows for applying techniques from algebraic topology. In this way, the notion of persistent homology was introduced \cite{Ca}.

The persistent homology plays an important role in TDA. The persistence diagrams are used to characterize persistent homology and thus to describe geometric properties of data. The set of all persistence diagrams can be endowed with different metrics. The most known are the Wasserstein metric and bottleneck metric.

The spaces of persistence diagrams are object of considerations in numerous publications (see, e.g., \cite{MMH,CCO,LOC,PMK,TS}). In particular, in \cite{PMK} a characterization theorem for compact subsets in the space of persistence diagrams is proved.

It is proved in \cite{BLPY} that the space of persistent diagrams is of infinite asymptotic dimension in the sense of Gromov.  This concerns the  Wasserstein metric on the set of persistence diagrams. Answering a question from \cite{BLPY} we prove an analogous result for the bottleneck metric on this set.

As we remark below, the set of all persistence diagrams is nothing but the infinite symmetric power of the upper (positive) half-plane. In this note we consider the strong (direct limit) topology on this set. One of our results is that the space of the persistence diagrams with this topology is homeomorphic to the countable direct limit of the euclidean spaces.

\section{Preliminaries}

\subsection{Persistence diagrams} Let $$\Delta=\{(x,y)\in\mathbb R^2_+\mid x= y\},\ \hat X=\{(x,y)\in\mathbb R^2_+\mid x\le y\},$$ and $X=\hat X\setminus\Delta$. For any $n\in\mathbb N$, let $\hat X_n=\{(x,y)\in \hat X\mid y\le n\}$, $X_n=\hat X_n\setminus\Delta$.

A {\em persistence diagram} is a function $\mu\colon X\to \mathbb Z_+$ such that $\mu(a)=0$ for all but finitely many $a\in \hat X$ and $\mu(a)=0$ for all $a\in\Delta$. The {\em support} of $\mu$ is the set $\mathrm{supp}(\mu)=\{a\in \hat X\mid \mu(a)>0\}$.

By $\mathcal D$ we denote the set of all persistence diagrams. Given $n\in\mathbb N$, we denote by $\mathcal D_n$ the set of all $\mu\in\mathcal D$ such that $|\mathrm{supp}(\mu)|\le n$.

\subsection{Bottleneck distance} Let $\mu\in \mathcal D$. A {\em sequential representation} of $\mu$ is a finite sequence $(a_1,\dots,a_k)$ such that the following are satisfied:
\begin{enumerate}
\item for every $a\in \mathrm{supp}(\mu)$, $|\{i\le m\mid a=a_i\}|=\mu(a)$;
\item if $a_i\notin \mathrm{supp}(\mu)$, then $a_i\in\Delta$.
\end{enumerate}
The number $k$ is said to be the {\em length} of the representation  $(a_1,\dots,a_k)$.

By $S_k$, the group of permutations of the set $\{1,\dots,k\}$ is denoted.

Let $\mu,\nu\in\mathcal D$. We define
\begin{align*} d(\mu,\nu)=&\inf\{\min\{\max\{\rho(a_i,b_{\sigma(i)})\mid 1\le i\le k\}\mid \sigma\in S_k\} \\ &\mid (a_1,\dots,a_k),\  (b_1,\dots,b_k) \text{ are sequential representations of }\\ &\mu\text{ and }\nu\text{ respectively},\ k\in\mathbb N\}.\end{align*}

(The assignment $a_i\mapsto b_{\sigma(i)}$, $i=1,\dots,k$, is said to be a {\em matching} (see, e.g., \cite{CSGO} for details).

The function $d$ is known to be a metric on $\mathcal D$ (the  bottleneck metric; see, e.g., \cite{CSGO}).

\subsection{Space $\mathbb R^\infty$} Recall that the direct limit of the increasing sequence of topological spaces $X_1\subset X_2\subset \dots$ (here $X_n$ is a subspace of $X_{n+1}$, for each $n$) is the set $X=\cup_{n=1}^\infty X_n$ endowed with the strongest topology inducing the original topology on each $X_n$. The obtained topological space is denoted by $\varinjlim X_n$.

We identify every $(x_1,\dots,x_n)\in \mathbb R^n$ with $(x_1,\dots,x_n,0)\in \mathbb R^{n+1}$. Thus, $\mathbb R^n$ is regarded as a subspace in $\mathbb R^{n+1}$. We denote by $\mathbb R^\infty$ the direct limit of the sequence $\mathbb R\subset \mathbb R^2\subset \mathbb R^3\subset\dots$.

A characterization theorem for the space $\mathbb R^\infty$ is proved by K. Sakai \cite{Sa}.

\begin{thm}[Characterization Theorem for $\mathbb R^\infty$]\label{t:sakai} Let $X$ be a countable direct limit of finite-dimensional compact metrizable spaces. The following are equivalent.
\begin{enumerate}
\item $X$ is homeomorphic to $\mathbb R^\infty$;
\item for every finite-dimensional compact metrizable pair $(A,B)$ and every embedding $f\colon B\to X$ there exists an embedding $\bar f\colon A\to X$ that extends $f$.
\end{enumerate}
\end{thm}
\subsection{Asymptotic dimension} Let $Y$ be a metric space. A family $\mathcal A$ of subsets of $X$ is said to be {\em uniformly bounded} if
$\sup\{\mathrm{diam}(A)\mid A\in\mathcal A\}<\infty$. Given $D>0$, we say that a family $\mathcal A$ is {\em $D$-disjoint} if, for every distinct $A,B\in\mathcal A$, $\mathrm{dist}(A,B)\ge D$

We say that the asymptotic dimension of $Y$ does not exceed $n$, if, for any $D>0$, there exists a uniformly bounded cover $\mathcal U$ of $Y$ such that $\mathcal U=\cup_{i=0}^n\mathcal U_i$, where every family $\mathcal U_i$ is $D$-disjoint, $i=0,\dots,n$. The notion of asymptotic dimension is defined by M. Gromov \cite{Gr}. See, e.g., \cite{BD} for properties of the asymptotic dimension.

\section{Main results}

\subsection{Strong topology on the space of persistence diagrams}
For every $n\in\mathbb N$, let $$\mathcal D_n=\{\mu\in \mathcal D\mid |\mathrm{supp}(\mu)|\le n\text{ and }\mathrm{supp}(\mu)\subset X_n\}$$ and $\mathcal D_\infty=\varinjlim\mathcal D_n$.

\begin{thm}\label{t:infty} The space $\mathcal D_\infty$ is homeomorphic to $\mathbb R^\infty$.
\end{thm}

\begin{proof} For any $n\in \mathbb N$, define a map $\xi_n\colon \hat X^n_n\to\mathcal D_n$ as follows: $\xi_n(a_1,\dots,a_n)(a)=|\{i\mid a=a_i\}|$, for all $a\in X_n$. Note that this map is clearly continuous and it admits a factorization $\xi_n=\xi'_n\xi''_n$, where $\xi''\colon \hat X^n_n\to (\hat X_n/(\hat X_n\cap \Delta)_n)^n$ is the factorization map, where $*_n$ stands for $\hat X_n\cap \Delta$ (actually, $\xi''_n=q^n$, where $q\colon \hat X^n\to \hat X_n/(\hat X_n\cap \Delta)_n$ is the factorization map). Therefore, $\mathcal D_n$ is the orbit space of the action of the group $S_n$ on the space $(\hat X_n/(\hat X_n\cap \Delta))^n$ by permutation of coordinates. In other words, $\mathcal D_n$ is homeomorphic to the $n$th symmetric power $SP^n(\hat X_n/(\hat X_n\cap \Delta))$. The orbit containing $(x_1,\dots,x_n)$ will be denoted by $[x_1,\dots,x_n]$.

We denote by $*_n$ the point $q(\hat X_n\cap \Delta)\in \hat X_n/(\hat X_n\cap \Delta)$. Identifying $*_n$ with $*_{n+1}$, one can consider $\hat X_n/(\hat X_n\cap \Delta)$ as a subset of $\hat X_{n+1}/(\hat X_{n+1}\cap \Delta)$. Then identifying $[x_1,\dots,x_n]\in SP^n(\hat X_n/(\hat X_n\cap \Delta))$ with $[x_1,\dots,x_n,*_{n+1}]\in SP^{n+1}(\hat X_{n+1}/(\hat X_{n+1}\cap \Delta))$ we finally obtain that $\mathcal D_\infty$ is homeomorphic to the space $$\varinjlim  SP^n(\hat X_n/(\hat X_n\cap \Delta))=SP^\infty (\varinjlim \hat X_n/(\hat X_n\cap \Delta).$$
The latter space is known as the infinite symmetric power construction \cite{DT}.

For every $n\in \mathbb N$, the space $\hat X_n/(\hat X_n\cap \Delta)$ is homeomorphic to the 2-dimensional disc. Therefore, the space $SP^n(\hat X_n/(\hat X_n\cap \Delta))$ is a finite dimensional absolute retract (AR), see \cite{Wa}.
Similarly as in \cite{Za} one can prove that the space $SP^\infty (\varinjlim \hat X_n/(\hat X_n\cap \Delta)$ (and also the space $\mathcal D_\infty$) is homeomorphic to $\mathbb R^\infty$.

For the sake of completeness, we provide here a proof based on Sakai's Characterization Theorem \ref{t:sakai}. Let $(A,B)$ be a finite-dimensional compact metrizable pair an let $f\colon B\to\mathcal D_\infty$ be an embedding. Then there exists $n\in\mathbb N$ such that $f(B)\subset\mathcal D_n$. As we already remarked, $\mathcal D_n$ is an absolute retract. Thus, there exists an extension $g\colon A\to\mathcal D_n$ of $f$.

We denote by $\alpha\colon A\to A/B$ the quotient map. Since the quotient space $A/B$ is finite-dimensional, there exists an embedding $i\colon A/B\to [0,1]^m$, for some $m$. Without loss of generality we assume that $\alpha(A/B)=0$, $\alpha(A\setminus B)\subset (0,1)^m$, and $m>n$. Write $i(a)=(i_1(a),\dots,i_m(a))$.

Now, given $x\in A$, write $g(x)$ as $[g_1(x), \dots, g_n(x)]$. Then define \begin{align*}\bar f(x)=&[g_1(x), \dots, g_n(x),\\ & (1,1+i_1(\alpha(x))),\dots, (2n+1,2n+1+i_1(\alpha(x))),\\ & (2n+2, 2n+2+i_2(\alpha(x)),\dots, (4n+3,4n+3+i_2(\alpha(x))),\\ &\dots \\
 &((m-1)(2n+1)+1,(m-1)(2n+1)+1+i_m(\alpha(x))),\dots, \\ & ( m(2n+1)+1,m(2n+1)+1+i_m(\alpha(x)))].\end{align*}
First, note that $\bar f$ is well defined and continuous. If $x\in B$, then $i_k(\alpha(x))
=0$, $k=1,\dots,n(2n+1)$, and therefore  $$\bar f(x)=[g_1(x), \dots, g_n(x), (1,1),\dots, (n(2n+1),n(2n+1)))]=f(x),$$ because of our identifications.

Clearly, $\bar f$ is well-defined and continuous. Since $A$ is compact, in order to prove that $\bar f$ is an embedding it is sufficient to prove that $\bar f$ is injective. Let $x,y\in A$, $x\neq y$. If $x\in B$ and $y\in A\setminus B$, then $|\mathrm{supp}(\bar f(x))|\le n$ and, since $i_k(\alpha(y))\neq0$, for some $k\le m$, we conclude that $|\mathrm{supp}(\bar f(y))|\ge n+1$. If $x,y\in A\setminus B$, then there exists $k\le m$, such that
$i_k(\alpha(x))\neq i_k(\alpha(y))$. Since $|\mathrm{supp}(g(x))|\le n$, we see that there exists $j\le n+1$ such that $$(j(2n+1)+1,j(2n+1)+1+ i_j(\alpha(x))\in\mathrm{supp}(\bar f(x))\setminus \mathrm{supp}(\bar f(y))$$ and therefore $\bar f(x)\neq\bar f(y)$. The other cases being treated similarly, this proves injectivity of $\bar f$. By Theorem \ref{t:sakai}, $\mathcal D_\infty$ is homeomorphic to $\mathbb R^\infty$.

\end{proof}

\subsection{Asymptotic dimension of the space of persistence diagrams} We consider the set $\mathcal D$ of all persistence diagrams with the bottleneck metric. Given $n\in\mathbb N$, consider the set \begin{align*}K_n=&\{[(n,n+(n+1)+t_1),(2n+1, 2n+1+(n+1)+t_2),\dots,\\&(n-1)n+(n-2), (n-1)n+(n-2)+(n+1)+t_{n-1})], \\& (n^2+(n-1),n^2+(n-1)+(n+1)+t_n )]\mid (t_1,\dots,t_n)\in[0,n]^n\}\end{align*} Clearly, $K_n$ is isometric to the cube $[0,n]^n$ endowed with the $l_\infty$-metric. This easily follows from the observation that every matching between two points from $K_n$ realizing the bottleneck distance consists of pairs of points from $X$ so that each pair lies on a vertical line.

Since $n$ can be chosen arbitrarily large, the known properties of the asymptotic dimension (see, e.g., \cite{BD}) imply the following.

\begin{thm} $\mathrm{asdim}\, \mathcal D=\infty$.
\end{thm}

\section{Remarks}

We conjecture that an analog of Theorem \ref{t:infty} can be proved for the space $\mathcal D_\infty=\varinjlim \mathcal D_n$ in the case when every $\mathcal D_n$ is endowed with the Wasserstein metric (also Wasserstein $p,q$ metric considered in \cite{BLPY}).

Since many spaces of persistence diagrams are infinitely-dimensional, one can expect that the methods of infinite-dimensional topology, in particular, the theory of infinite-dimensional manifolds, will be useful in their investigations.

In \cite{CCO}, the space $\mathcal D_N^b$ of bounded persistent diagrams with less than $N$ points is mentioned. We consider the space $\tilde{\mathcal D}_N$ of exactly $N$ points (taking into account the multiplicities), $N\in\mathbb N$. Having in mind the mentioned identification of persistence diagrams and symmetric powers one can derive from \cite[Theorem 4.5]{Wa} that the space $\tilde{\mathcal D}_N$ is homeomorphic to the euclidean space $\mathbb R^{2N}$. This leads to the question of description of topology of the subspace $\tilde{\mathcal D}_{\le N}=\cup_{i\le N}\tilde{\mathcal D}_{i}$ of $\mathcal D$.

In this note we restricted ourselves with persistence diagrams of finite support. In some publications, persistence diagrams with countably many points are also considered. In particular, it is known that the latter spaces are complete in the Wasserstein metric. In the subsequent publications we are going to consider the geometry of the complete spaces of persistence diagrams and some of their subspaces.

\end{document}